\documentclass[12pt]{article}
\usepackage{amsmath,amssymb,amsbsy,amsfonts,amsthm,latexsym,amsopn,amstext,%
                                               amsxtra,euscript,amscd}
\begin{document}

\newtheorem{theorem}{Theorem}
\newtheorem{lemma}[theorem]{Lemma}
\newtheorem{claim}[theorem]{Claim}
\newtheorem{cor}[theorem]{Corollary}
\newtheorem{prop}[theorem]{Proposition}
\newtheorem{definition}{Definition}
\newtheorem{question}[theorem]{Open Question}
\newtheorem{conj}[theorem]{Conjecture}
\newtheorem{prob}{Problem}
\newtheorem{algorithm}[theorem]{Algorithm}

\def\squareforqed{\hbox{\rlap{$\sqcap$}$\sqcup$}}
\def\qed{\ifmmode\squareforqed\else{\unskip\nobreak\hfil
\penalty50\hskip1em\null\nobreak\hfil\squareforqed
\parfillskip=0pt\finalhyphendemerits=0\endgraf}\fi}

\def\cA{{\mathcal A}}
\def\cB{{\mathcal B}}
\def\cC{{\mathcal C}}
\def\cD{{\mathcal D}}
\def\cE{{\mathcal E}}
\def\cF{{\mathcal F}}
\def\cG{{\mathcal G}}
\def\cH{{\mathcal H}}
\def\cI{{\mathcal I}}
\def\cJ{{\mathcal J}}
\def\cK{{\mathcal K}}
\def\cL{{\mathcal L}}
\def\cM{{\mathcal M}}
\def\cN{{\mathcal N}}
\def\cO{{\mathcal O}}
\def\cP{{\mathcal P}}
\def\cQ{{\mathcal Q}}
\def\cR{{\mathcal R}}
\def\cS{{\mathcal S}}
\def\cT{{\mathcal T}}
\def\cU{{\mathcal U}}
\def\cV{{\mathcal V}}
\def\cW{{\mathcal W}}
\def\cX{{\mathcal X}}
\def\cY{{\mathcal Y}}
\def\cZ{{\mathcal Z}}

\def\fI{{\mathfrak I}}
\def\fJ{{\mathfrak J}}

\def\MNL{{\mathfrak M}(N;K,L)}
\def\VNL{V_m(N;K,L)}
\def\RNL{R(N;K,L)}

\def\MNm{{\mathfrak M}_m(N;K)}
\def\VNm{V_m(N;K)}

\def\Xm{\cX_m}

\def \C {{\mathbb C}}
\def \F {{\mathbb F}}
\def \L {{\mathbb L}}
\def \K {{\mathbb K}}
\def \Q {{\mathbb Q}}
\def \Z {{\mathbb Z}}

\def\barG{\overline{\cG}}
\def\\{\cr}
\def\({\left(}
\def\){\right)}
\def\fl#1{\left\lfloor#1\right\rfloor}
\def\rf#1{\left\lceil#1\right\rceil}

\newcommand{\pfrac}[2]{{\left(\frac{#1}{#2}\right)}}

\def\rem{\mathrm{\, rem~}}

\def \Prob{{\mathrm {}}}
\def\e{\mathbf{e}}
\def\ep{{\mathbf{\,e}}_p}
\def\epp{{\mathbf{\,e}}_{p^2}}
\def\em{{\mathbf{\,e}}_m}
\def\eps{\varepsilon}
\def\Res{\mathrm{Res}}
\def\vec#1{\mathbf{#1}}

\def \li {\mathrm {li}\,}

\def\mand{\qquad\mbox{and}\qquad}

\newcommand{\comm}[1]{\marginpar{%
\vskip-\baselineskip 
\raggedright\footnotesize
\itshape\hrule\smallskip#1\par\smallskip\hrule}}


\title{\Large\bf On Vanishing Fermat Quotients and a Bound of the Ihara Sum}

\author{{\sc Igor E. Shparlinski}\\
{Department of Computing, Macquarie University}\\
{ Sydney, NSW 2109,
Australia}\\
{\tt igor.shparlinski@mq.edu.au} } 
\date{\today}

\maketitle

\begin{abstract} We  improve an estimate of
A.~Granville (1987) on the
number of vanishing Fermat quotients $q_p(\ell)$ 
modulo a prime $p$ when $\ell$ runs through 
primes $\ell \le N$.  We use this bound  to obtain an 
unconditional improvement of the 
conditional (under the Generalised Riemann Hypothesis)
estimate of Y.~Ihara (2006) on a certain sum, related 
to vanishing Fermat quotients. In turn this sum  
appears in the study of the index of certain 
subfields of of 
cyclotomic fields $\Q(\exp(2 \pi i/p^2))$. 
\end{abstract}


\section{Introduction}

For a prime $p$ and an integer $u$ with $\gcd(u,p)=1$ 
we define the {\it Fermat quotient\/} $q_p(u)$
as the unique integer
with 
$$
q_p(u) \equiv \frac{u^{p-1} -1}{p} \pmod p, \qquad 0 \le q_p(u) \le p-1.
$$
We also define $q_p(u) = 0$ for $u \equiv 0 \pmod p$.

Fermat quotients appear and play a 
major role in various questions of
 computational and algebraic number theory
and thus have  been studied in a number of works, 
see, for 
example,~\cite{BFKS,ErnMet,Fouch,Gran1,Gran2,
Ihara,Len,OstShp}
and references therein. Amongst other properties,  
the $p$-divisibility of Fermat
quotients $q_p(a)$ by $p$ is important for many  applications
and in particular, the smallest value $\ell_p$ of $u\ge 1$ with
$q_p(u) \ne  0$, has been studied in a number of works,
see~\cite{BFKS,ErnMet,Fouch,Gran1,Len}.
For example, in~\cite{BFKS}, improving the previous 
estimate $\ell_p  = O\( (\log p)^2\)$  
of Lenstra~\cite{Len} (see also~\cite{Fouch,Gran2,Ihara}), the
following  bounds have been given:
$$
\ell_p \le \left\{\begin{array}{lll}
 (\log p)^{463/252 + o(1)}  &\quad \text{for all primes }\ p, \\
 (\log p)^{5/3 + o(1)}  &\quad \text{for almost all primes}\ p, 
\end{array}\right.
$$
(where ``almost all primes $p$'' means for all primes $p$ but a set 
of relative density zero).

Here we use some results of~\cite{BFKS}, combined with the
approach of Granville~\cite{Gran0} 
to obtain new estimates  on the cardinality of the sets
\begin{eqnarray*}
\cQ_p(N) & = &\{n \le N~:~q_p(n) = 0\}, \\
\cR_p(N) & = & \{\ell \le N~:~\ell~\text{prime},\ q_p(\ell) = 0\},
\end{eqnarray*}
which for small $N$ improve that of~\cite{Gran0}.
We  apply these improvements to study  the sums
$$
S_p = \sum_{n \in\cQ_p(p)} \frac{\Lambda(n)}{n}
$$ 
introduced by Ihara~\cite{Ihara}, 
where, as usual,
$$
\Lambda(n)=
\begin{cases}
  \log \ell,   & \quad\text{if}~n~\text{is a power of a prime}~\ell, \\
       0,   & \quad\text{otherwise},
\end{cases}
$$
be the {\it von Mangoldt function\/}.

We note that in~\cite[Corollary~7]{Ihara}, 
under the {\it Generalised Riemann Hypothesis\/}, the bound
\begin{equation}
\label{eq:Ihara}
S_p \le 2 \log \log p + 2 + o(1)
\end{equation}
as $p\to \infty$, has been obtained. Here we give
an unconditional proof of a stronger bound. 

Throughout the paper,  
the implied constants in the symbols `$O$',
and `$\ll$' may occasionally  depend on the  real positive  parameter
$\alpha$ and are absolute otherwise (we recall that 
the notation $U \ll V$ is 
equivalent to   $U = O(V)$).

\section{Preparations}

We recall that  for any integers $m$ and $n$ 
with $\gcd(mn,p) = 1$ we have
\begin{equation}
\label{eq:add struct}
q_p(mn) \equiv q_p(m) + q_p(n) \pmod p,
\end{equation} 
see, for example,~\cite[Equation~(2)]{ErnMet}.

Let $\cG_p$ be   the group of the
$p$th power residues in the unit group $\Z_{p^2}^*$  of  
the residue ring $\Z_{p^2}$ 
modulo $p^2$.

\begin{lemma}
\label{lem:Gp}
For any $u \in \Z_{p^2}^*$ the conditions
$q_p(u) = 0$ and $u \in \cG_p$ are equivalent. 
\end{lemma}

\begin{proof}
Clearly $q_p(u)=0$ for $u \in \Z_{p^2}^*$  
is equivalent to  $u^{p-1} \equiv 1 \pmod {p^2}$, which in 
turn is equivalent to $u \in \cG_p$.
\end{proof}

Let $T_p(K)$ be the number of $w \in [1, K]$
such that their residues modulo $p^2$ belong to
$\cG_p$. The following estimate
follows immediately from~\cite[Equation~(12)]{BFKS}.

\begin{lemma}
\label{lem:Sols}
For any fixed  
$$
\alpha >  \frac{463}{252},
$$
and 
$$
K \ge p^{\alpha}
$$
we have
$$
T_p(K) \ll K/p. 
$$ 
\end{lemma}

Let $\tau_s(n)$ be the number of
representations of $n$ as a product of $s$ positive integers:
$$
\tau_s(n)=\#\big\{(n_1,\ldots,n_s)\in{\mathbb N}^s\,
|\,n=n_1n_2\ldots n_s\big\}.
$$

We also need the following upper bound from~\cite{Usol}:

\begin{lemma}
\label{es}
Uniformly over $n$ and $s$ we have
$$
\tau_s(n)\le
\exp\(\frac{(\log n)(\log s)}{\log\log n}
\(1+ O\(\frac{\log\log\log n+\log s}{\log\log n}\)\)\).
$$ 
\end{lemma}

In particular, we have:
\begin{cor}
\label{cor:taus}
If $s = (\log n)^{o(1)}$ then 
$$
\tau_s(n)\le n^{o(1)}. 
$$ 
as $n \to \infty$.
\end{cor}

\section{Distribution of vanishing Fermat quotients}

Here we estimate the cardinality of the sets 
$\cQ_p(N)$ and  $\cR_p(N)$. 
For large values of $N$, namely 
for $N \ge p^\alpha$ with $\alpha > 463/252$
such a bound is given by Lemma~\ref{lem:Sols}.
However here we are mostly interested in small values 
of $N$.

We note that Granville~\cite{Gran0} has given a 
bound on the cardinality of the set $\cR_p(N)$. 
Namely, it is shown in~\cite{Gran0} that for $u=1,2,\ldots$
\begin{equation}
\label{eq:Gran bound}
\# \cR_p(p^{1/u}) \le u p^{1/2 u}.
\end{equation}
We note that the  argument used in the proof
of~\eqref{eq:Gran bound} can be used to estimate 
$\# \cR_p(p^{1/u})$ for any $u\ge 1$. 

We derive now upper bounds on  $\# \cQ_p(N)$ 
and $\# \cR_p(N)$  that 
improve~\eqref{eq:Gran bound}.

\begin{theorem}
\label{thm:Qp} For any fixed  
$$
\alpha >  \frac{463}{252},
$$
for  $1 \le u  = (\log p)^{o(1)}$,
where  
$$
u = \frac{\log p}{\log N},
$$
we have
$$
\# \cQ_p(N) \ll u  N  p^{-(1+o(1))/\rf{\alpha u}}.
$$
as $p\to \infty$. 
\end{theorem}

\begin{proof} We put 
$$
s = \rf{\alpha u}.
$$ 
We consider $\(\# \cQ_p(N)\)^s$
products $n = n_1\ldots n_s$
where $(n_1, \ldots, n_s) \in  \cQ_p(N)^s$. 
By~\eqref{eq:add struct} we see that 
$$
q_p(n) =  q(n_1)\ldots q_p(n_s) = 0.
$$
Besides, using Corollary~\ref{cor:taus} we see that each 
$n \le N^s <p^{\alpha + 1}$ has at most 
$$
\tau_s(n) = p^{o(1)}
$$ 
such representations.
We also note that $N^s \ge  p^{\alpha}$.
Therefore, combining Lemmas~\ref{lem:Gp} and~\ref{lem:Sols}, we derive
$$
\(\# \cQ_p(N)\)^s \le T_p(N^s)p^{o(1)} \le N^sp^{-1+o(1)},  $$
which implies  the desired result.
\end{proof}

\begin{cor}
\label{cor:Qp}
If  
$$ 
\frac{\log p}{\log N} = (\log p)^{o(1)} \mand \frac{\log p}{\log N} \to \infty$$ 
then 
$$
\# \cQ_p(N) \le  N^{211/463 + o(1)} 
$$
as $p\to \infty$. 
\end{cor}

For the set $\cR_p(N)$ we have a bound in a wider 
range of $u$.

\begin{theorem}
\label{thm:Rp}
For any fixed  
$$
\alpha >  \frac{463}{252},
$$
for  $u \ge 1$,
where  
$$
u = \frac{\log p}{\log N},
$$
we have
$$
\# \cR_p(N) \ll u  N  p^{-1/\rf{\alpha u}} 
$$
as $p\to \infty$. 
\end{theorem}

\begin{proof} The proof is the same as that of 
Theorem~\ref{thm:Qp} except that instead of 
Corollary~\ref{cor:taus}  we note that there are at most 
$s!$ products of $s$ primes 
$\ell_1\ldots \ell_s$ that take the same value.
So, we derive
$$
\(\# \cR_p(N)\)^s \ll s! T_p(N^s) \ll s! N^s p^{-1}, 
$$
and  the result now follows.
\end{proof}

\begin{cor}
\label{cor:Rp}
If  $N < p$ and 
$$ 
\frac{\log p}{\log N} \to \infty$$ 
then 
$$
\# \cR_p(N) \le  N^{211/463+ o(1)} \log p 
$$
as $p\to \infty$. 
\end{cor}

\section{Ihara sums}

First we consider approximations of $S_p$ by partial sums
$$
S_p(N) = \sum_{n \in\cQ_p(N)} \frac{\Lambda(n)}{n}.
$$

\begin{theorem}
\label{thm:SpN}
For $N = p^{o(1)}$ we have
$$
S_p  = S_p(N)  + O(N^{-252/463 + o(1)} \log p)
$$
as $p\to \infty$. 
\end{theorem}

\begin{proof} 
Clearly, we have
\begin{equation}
\label{eq:tail}
S_p - S_p(N)  =  \sum_{\substack{\ell  > N\\\ell\in \cR_p(p)}}  \frac{\log \ell}{\ell} + O(N^{-1} \log N).
\end{equation}

We now see from Corollary~\ref{thm:Qp} that for any 
$$
L < N^{3}
$$
we have
\begin{equation}
\label{eq:small L}
\begin{split}
\sum_{\substack{2L \ge \ell  > L\\ \ell\in \cR_p(p)}} \frac{\log \ell}{\ell} & \le \frac{\log L}{L} \sum_{\ell\in \cR_p(2L)} 1\\
 & \le \frac{\log L}{L}  L^{211/463 + o(1)} \log p
= L^{-252/463 + o(1)} \log p.
\end{split}
\end{equation} 

For  
$$
p \ge L > N^{3}
$$
we choose 
$$
\alpha = \frac{463}{251}
$$
and note that for $u \ge 1$
we have 
$$
\rf{\alpha u} \le  \frac{3}{2}  \alpha u.
$$
Thus  Theorem~\ref{thm:Rp} implies the bound 
$$
\# \cR_p(L) \ll   L^{1-2/3\alpha} \log p \ll 
 L^{2/3} \log p.
$$
Hence in the above range, we have
\begin{equation}
\label{eq:large L}
\begin{split}
\sum_{\substack{2L \ge \ell  > L\\ \ell\in \cR_p(p)}} \frac{\log \ell}{\ell} & \le \frac{\log L}{L} \sum_{\ell\in \cR_p(2L)} 1\\
 & \le \frac{\log L}{L}  L^{2/3} \log p
= L^{-1/3 + o(1)} \log p.
\end{split}
\end{equation} 
Thus covering the range $[N, p]$ by dyadic intervals
of the form $[L,2L]$ and using the bounds~\eqref{eq:small L},
and~\eqref{eq:large L} we  derive
$$
\sum_{\substack{\ell  >N\\\ell\in \cR_p(p)}} \frac{\log \ell}{\ell} \le N^{-252/463 + o(1)} \log p, 
$$
which after the substitution in~\eqref{eq:tail} implies
the desired estimate.
\end{proof}

Since by the Mertens formula (see, 
for example,~\cite[Equation~(2.14)]{IwKow})
$$
S_p(N) \le \sum_{n \le  N} \frac{\Lambda(n)}{n} = \log N + O(1),
$$
we derive from Theorem~\ref{thm:SpN}:

\begin{cor}
\label{cor:Sp1}
For $N = p^{o(1)}$ we have
$$
S_p  \le  \log N  + O(N^{-252/463 + o(1)} \log p + 1)
$$
as $p\to \infty$.
\end{cor}


We now obtain an unconditional improvement of the conditional
estimate~\eqref{eq:Ihara}.

\begin{cor}
\label{cor:Sp2}
We have
$$
S_p \le  \(463/252 + o(1)\) \log \log p 
$$
as $p\to \infty$. 
\end{cor}

\begin{proof}
Taking $N = \rf{ (\log p)^{\alpha}}$ with $\alpha >463/252$
in the bound of Corollary~\ref{cor:Sp1} leads to the estimate
$$
S_p \le \alpha \log \log p +  O(1).
$$
Since  $\alpha$ is arbitrary, the result now follows.
\end{proof}


\section{Index of some  subfields of cyclotomic fields}

We recall that  the index $I(\K)$ of an algebraic 
number field $\K$ is the greatest 
common divisor of indexes $\left[\cO_\K:\Z[\xi]\right]$
taken over all $\xi \in \cO_\K$, where $\cO_\K$ is the ring of
integers of $\K$. 

As in~\cite{Ihara}, we denote by  $I_p$ the index of the field $\K_p$, which 
is the unique cyclic extension of degree $p$ over $\Q$
that is 
contained in the cyclotomic field $\Q(\exp(2 \pi i/p^2))$.

It has been shown in~\cite[Proposition~4~(i)]{Ihara} that 
under the Generalised Riemann Hypothesis the bound
\begin{equation}
\label{eq:I Cond}
\log I_p \le (1+o(1))p^2 \log \log p
\end{equation}
holds as $p\to \infty$. Also~\cite[Proposition~5]{Ihara}
gives an unconditional but weaker bound
$$
\log I_p \le (1/4+o(1))p^2 \log p.
$$

We use Corollary~\ref{cor:Sp2} to obtain an unconditional
improvement of~\eqref{eq:I Cond}.

\begin{theorem}
\label{thm:Ip}
We have
$$
\log I_p \le \(\frac{463}{504}+o(1)\)p^2 \log \log p
$$
as $p\to \infty$. 
\end{theorem}

\begin{proof} By~\cite[Equation~(2.4.1)]{Ihara} we have
\begin{equation}
\label{eq:I Expl}
\log I_p = \sum_{n \in \cQ_p(p)} \alpha_p(n) \Lambda(n),
\end{equation}
where 
$$
\alpha_p(n) =
\fl{\frac{p}{n}}\(p - \frac{1}{2}n - \frac{1}{2}\fl{\frac{p}{n}}n\).
$$
Since
$$
\alpha_p(n) =\fl{\frac{p}{n}}
\(p - \frac{1}{2}n\(1  +\fl{\frac{p}{n}}\)\) \le \fl{\frac{p}{n}}
\frac{p}{2} \le \frac{p^2}{2n}, 
$$
we see from~\eqref{eq:I Expl} that 
$$
\log I_p \le \frac{p^2}{2} S_p.
$$
Using  Corollary~\ref{cor:Sp2}, we conclude the proof.
\end{proof}

One certainly expects that $I_p$ is much smaller,
than the bound given in Theorem~\ref{thm:Ip},
however no unconditional lower bound  seems to be known
(see~\cite[Proposition~4~(ii)]{Ihara} for a 
conditional estimate).  

 \section*{Acknowledgement}
The author is  very grateful to Yasutaka Ihara, Sergei Konyagin 
and Arne Winterhof for their comments.

During the preparation of this work the author was supported in part by 
the  Australian Research Council  Grant~DP1092835.


\begin{thebibliography}{99}
 

\bibitem{BFKS} J.~Bourgain, K. Ford, S. V. Konyagin and
I. E. Shparlinski,
`On the divisibility of Fermat quotients',
{\it Michigan Math. J.\/}, {\bf 59} (2010), 313--328.

%
 


 

\bibitem{ErnMet} R. Ernvall and T. Mets{\"a}nkyl{\"a},
`On the $p$-divisibility of Fermat quotients',
{\it Math. Comp.\/}, {\bf  66} (1997),  1353--1365.

 

\bibitem{Fouch} W. L. Fouch{\'e},
`On the Kummer-Mirimanoff congruences',
{\it Quart. J. Math. Oxford\/}, {\bf 37} (1986),  257--261.

 

%

\bibitem{Gran0} A. Granville,
{\it Diophantine equations with varying exponents\/}, PhD 
Thesis, QueenÕs University, Kingston, Ontario, Canada, 1987.

\bibitem{Gran1} A. Granville,
`Some conjectures related to Fermat's Last Theorem',
{\it Number Theory\/}, 
 Walter de Gruyter, NY, 1990, 177--192.

\bibitem{Gran2} A. Granville,
`On pairs of coprime integers with no large prime factors',
{\it Expos. Math.\/}, {\bf 9} (1991), 335--350.
 
%

\bibitem{Ihara} Y. Ihara, `On the Euler-Kronecker constants of
global fields and primes with small norms',
{\it Algebraic Geometry and Number Theory\/}, Progress in Math.,
Vol. 850, Birkh{\"a}user, Boston, Cambridge, MA, 2006, 407--451.
%
 \bibitem{IwKow} H. Iwaniec and E. Kowalski,
{\it Analytic number theory\/}, Amer.  Math.  Soc.,
Providence, RI, 2004.


\bibitem{Len} H. W. Lenstra, `Miller's primality test',
{\it  Inform. Process. Lett.\/}, {\bf 8} (1979), 86--88.
 
\bibitem{OstShp}
A. Ostafe and I.~E.~Shparlinski,  
`Pseudorandomness and dynamics of Fermat quotients',
{\it SIAM J. Discr. Math.\/},   {\bf 25} (2011),  50--71.

 
%
%
%


\bibitem{Usol} L.~P.~Usol'tsev,
`On an estimate for a multiplicative function',
{\it   Additive problems in number
theory\/},  Kuybyshev. Gos. Ped. Inst.,
         Kuybyshev, 1985,  34--37 (in Russian).
 



 



\end{thebibliography}
\end{document}